\documentclass[11pt]{amsart}

\usepackage{amsmath,amssymb,amsthm,amsfonts,amscd} 
\usepackage[all]{xy}
\usepackage{tikz}
\usepackage{enumerate}
\usepackage{slashed}
\usetikzlibrary{matrix}
\usepackage{titletoc}
\usepackage{setspace}
\usepackage{mathtools}
\usepackage{url}

\newcommand{\vol}{\mathrm{Vol}}

\newcommand{\p}{\partial}
\newcommand{\bignatural}{\mathop{\mathchoice
  {\vcenter{\hbox{\LARGE$\natural$}}}
  {\vcenter{\hbox{\large$\natural$}}}
  {\vcenter{\hbox{\footnotesize$\natural$}}}
  {\vcenter{\hbox{\scriptsize$\natural$}}}
}\displaylimits}

\newtheorem{prop}{Proposition}
\newtheorem{cor}{Corollary}
\newtheorem{theorem}{Theorem}
\newtheorem*{theoremA}{Theorem A}
\newtheorem*{theoremB}{Theorem B}

\newtheorem*{corA}{Corollary A}
\newtheorem*{corB}{Corollary B}

\newtheorem{claim}{Claim}

\newtheorem*{definition}{Definition}
\newtheorem{lemma}{Lemma}

\newtheorem{conjecture}{Conjecture}

\begin{document}
\title[Desingularizing psc $4$-manifolds]{Desingularizing positive scalar 
curvature $4$-manifolds}

\author{Demetre Kazaras}
\address{
Department of Mathematics\\
Duke University \\
Durham, NC, 27708\\
USA
}
\email{demetre.kazaras@duke.edu}


  \date{\today}
\begin{abstract}
We show that the bordism group of closed $3$-manifolds with positive scalar curvature (psc) metrics 
is trivial by explicit methods. Our constructions are derived from scalar-flat K{\"a}hler ALE
surfaces discovered by Lock-Viaclovsky. Next, we study psc $4$-manifolds
with metric singularities along points and embedded circles. Our psc null-bordisms
are essential tools in a desingularization process developed by Li-Mantoulidis. This
allows us to prove a non-existence result for singular psc metrics on enlargeable
$4$-manifolds with uniformly Euclidean geometry. As a consequence, we obtain a positive mass
theorem for asymptotically flat $4$-manifolds with non-negative scalar curvature and low regularity.
\end{abstract}  
\maketitle
\renewcommand{\baselinestretch}{1.3}\normalsize
\vspace*{-10mm}

\begin{small}
\tableofcontents
\end{small}

\newpage

\section{Introduction and Main Results}

It is a fundamental fact in differential geometry that the $n$-dimensional torus $T^n$ cannot support
a smooth Riemannian metric with positive scalar curvature (psc). In dimension $2$ this result is
a simple consequence of the Gauss-Bonnet formula and in higher 
dimensions it is known as the Geroch Conjecture. 
For $3\leq n\leq 7$ this was
first proven by Schoen-Yau using a minimal hypersurface technique \cite{SY79}
and Gromov-Lawson later
ruled out psc tori in all dimensions using index theoretic methods \cite{GL80}. 
It is natural to wonder: does the non-existence of psc metrics on 
$T^n$ still hold
when one relaxes the smoothness assumption and allow for metric singularities? Without restricting
one's attention to a special class of singularities,
the answer is of course ``no'' in general. Inspired by constructions arising in Gromov's polyhedral 
comparison theory for psc metrics \cite{Gro}, Li-Mantoulidis introduced a class of 
singular metrics well-suited
for this non-existence question in \cite{LM}, presented a general conjecture, and gave a 
complete answer in dimension $3$. In the present paper, we consider the situation in
dimension $4$ where one encounters many new challenges.

Following \cite{LM}, a metric $g$ on a smooth manifold $M$ is said to be
{\em{uniformly Euclidean}}, denoted by $g\in L_E^\infty(M)$, if it may 
be pointwise bounded above
and below by some smooth metric.
In order to state the relevant conjecture in sufficient generality, recall that the 
Yamabe invariant of a closed $n$-manifold $M$ is given by
\begin{equation}
\sigma(M)=\sup_{C}\inf_{\tilde{g}\in C}\frac{\int_MR_{\tilde{g}}dvol_{\tilde{g}}}{\mathrm{Vol}_{\tilde{g}}(M)^{\frac{n-2}{n}}}
\end{equation}
where the supremum is taken over all conformal classes $C$ 
of smooth Riemannian metrics on $M$ 
and $R_{\tilde{g}}$ denotes the scalar curvature of $\tilde{g}$. This diffeomorphism 
invariant has a long and rich history. Its relevance to our present 
discussion: $\sigma(M)>0$ if and only if $M$
admits a psc metric. The following conjecture of Schoen originally appeared in \cite{LM}.
\begin{conjecture}\cite{LM}\label{con:schoen}
Let $M^n$ be a closed manifold with $\sigma(M)\leq0$ and let $S\subset M$ be a smooth, closed, embedded submanifold of codimension at least $3$. If $g$ is an $L_E^\infty(M)\cap C^\infty(M\setminus S)$ metric with $R_g\geq0$ on $M\setminus S$, then $g$ extends smoothly to a Ricci-flat metric on $M$.
\end{conjecture}
\noindent Notice that a metric $g\in L_E^\infty(M)\cap C^\infty(M\setminus S)$
may have unbounded curvature. Li-Mantoulidis have confirmed Conjecture \ref{con:schoen} 
in dimension $3$ \cite[Theorem 1.4]{LM}. 

In Corollary A, we make additional progress on Conjecture \ref{con:schoen} by applying
our main desingularization result, described in Theorem A. For a $4$-manifold $M^4$, we call a subset $S\subset M$ a {\emph{tolerable singular set}} if it is the image of a finite collection of smooth and disjoint embeddings of circles and points. We emphasize that our interest is in {\emph{metric singularities}} and the underlying manifolds we consider are assumed to be smooth.
\begin{theoremA}
Let $M^4$ be a closed oriented $4$-manifold and let $S\subset M$ be a tolerable singular set. Suppose $g$ is an $L^\infty_E(M)\cap C^{\infty}(M\setminus S)$ metric with $R_g>0$ on $M\setminus S$. Then there exists a smooth closed oriented psc manifold $(\overline{M},\overline{g})$ with a degree-$1$ map $F:\overline{M}\to M$. Moreover, there is a neighborhood $U\subset M$, containing
and retracting onto $S$, so that $F|_{\overline{M}\setminus F^{-1}(U)}$ is a conformal diffeomorphism.
\end{theoremA}

\begin{figure}[!htb]
\label{f:theoremA}
\includegraphics[scale=.4]{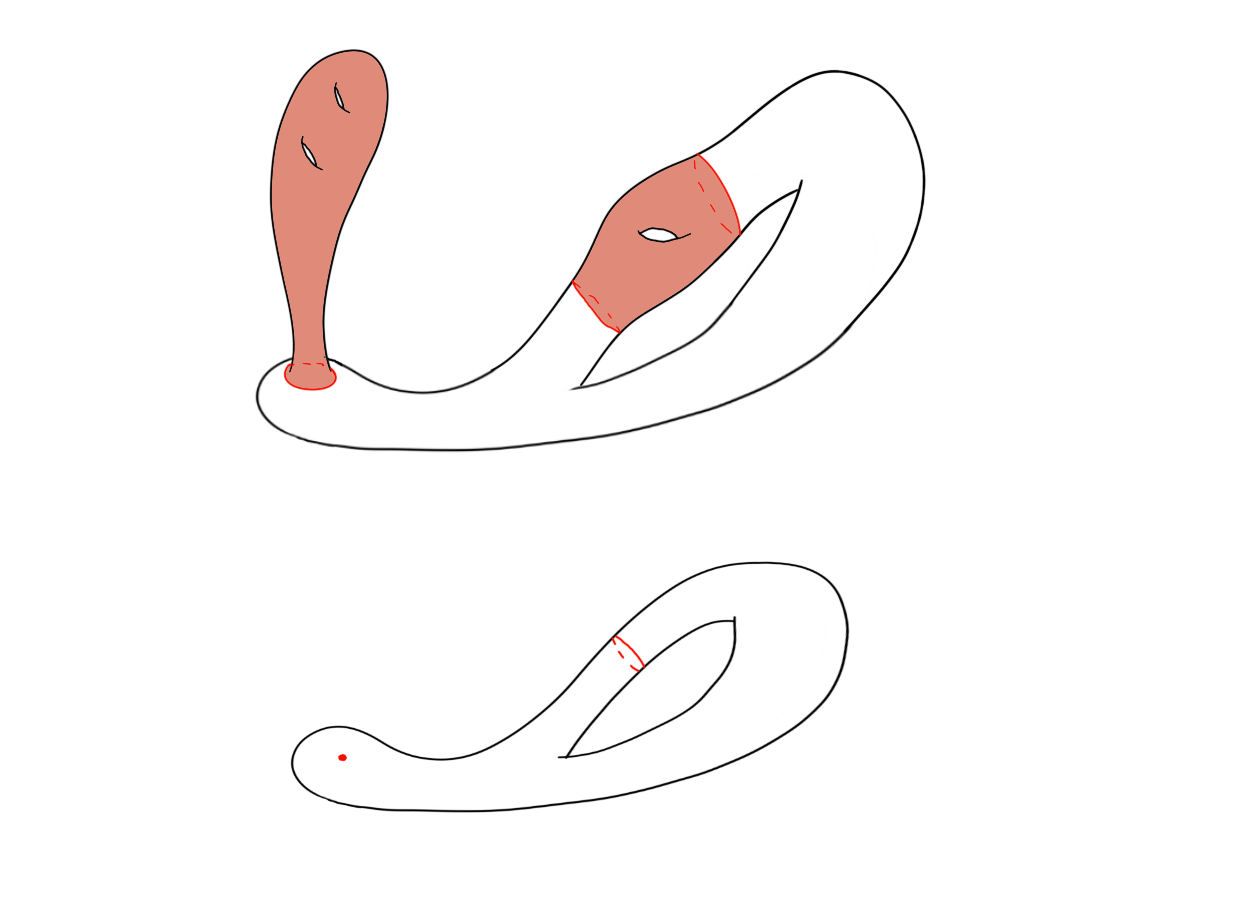}
\begin{picture}(0,0)
\put(-90,180){{\Large $(\overline{M},\overline{g})$}}
\put(-90,70){{\Large $(M,g)$}}
{\thicklines
\put(-70,170){\vector(0,-1){82}}
}
\put(-64,120){{\Large $F$}}
{\color{red}
\put(-260,70){{\large $S$}}
\put(-263,67){\vector(-2,-3){14}}
\put(-247,75){\vector(1,0){39}}
\put(-250,220){{\large $F^{-1}(S)$}}
\put(-253,217){\vector(-3,-2){35}}
\put(-210,217){\vector(1,-1){15}}
}
\end{picture}
\caption{The desingularization $F:\overline{M}\to M$ in Theorem A.}
\end{figure}
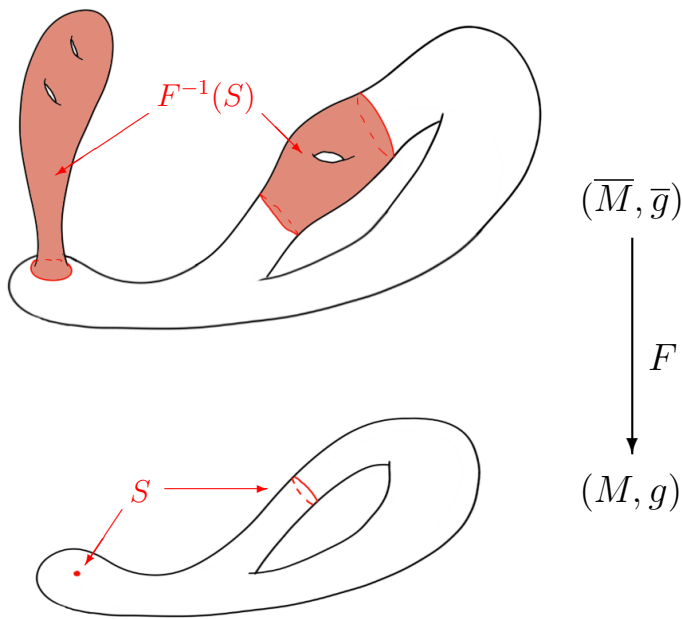
In Theorem A, we regard $(\overline{M},\overline{g})$ as a desingularization of $(M,g)$, pictured 
in Figure \ref{f:theoremA}. In general, 
$\overline{M}$ is not diffeomorphic to $M$, possibly having much larger second Betti number. 
This is in contrast to \cite{LM},
where Li-Mantoulidis show that point singularities in $3$-manifolds may be 
desingularized without changing the underlying manifold's topology.
While we cannot show that this change in topology is necessary, one 
can regard this feature of our construction
as a reflection of the added topological complexities one encounters in dimension $4$.

We use Theorem A to establish a non-existence 
result for singular psc metrics of on a class of $4$-manifolds with 
non-positive $\sigma$-invariant called
{\em{enlargeable manifolds}}. Recall that a closed oriented $n$-manifold $M^n$
is said to be enlargeable if, for every Riemannian metric on $M$ and $\varepsilon>0$, 
there is a cover of $M$ which admits an $\varepsilon$-contracting map of non-zero degree 
to the unit sphere.
The prototypical example is the torus where one may consider sufficiently large 
coverings and produce
highly compressive maps to the sphere.
This notion was originally introduced by Gromov-Lawson in \cite{GL80} in order 
to relate the fundamental group
to Dirac operator methods for spin manifolds, showing that enlargeable spin 
manifolds do not admit psc metrics. Cecchini-Schick have extended this result to non-spin manifolds in \cite{CS}. We review the relevant definitions and facts of enlargeable manifolds in Section \ref{s:en}.

\begin{corA}
Let $M^4$ be a closed oriented enlargeable $4$-manifold and let $S\subset M$ be a 
tolerable singular set.
Then $M$ cannot admit a metric $g$ in $L_E^\infty(M)\cap C^\infty(M\setminus S)$ with $R_g>0$ on $M\setminus S$.
Moreover, if $g$ in $L_E^\infty(M)\cap C^\infty(M\setminus S)$ satisfies $R_g\geq 0$ on $M\setminus S$,
then $g$ is Ricci-flat on $M\setminus S$.
\end{corA}
\noindent For context, we remark that while the class of 
enlargeable $4$-manifolds is large, there is a comparable number of 
non-enlargeable $4$-manifolds which do not admit psc metrics, due to the non-triviality 
of their Seiberg-Witten invariant. In Section \ref{s:disc} we discuss this and 
the status of Conjecture \ref{con:schoen} in dimension $4$.
We expect that Ricci-flat metrics $g$ on $M\setminus S$ smoothly 
extend over the singular set. 

As an application of Corollary A, we obtain a Riemannian positive mass theorem for singular metrics of non-negative scalar 
curvature on asymptotically flat $4$-manifolds. See Section \ref{s:ALEpre} for the relevant definitions.
\begin{corB}
Let $(M^4,g)$ be an asymptotically flat $4$-manifold with 
$g\in L^{\infty}(M)\cap C^{\infty}(M\setminus S)$
where $S$ is a tolerable singular set. 
If $R_g\geq0$, then $m(M,g)\geq0$.
\end{corB}
\noindent The proof of Corollary B is based on a well-known argument due to Lohkamp 
\cite{Loh} which reduces 
the Riemannian positive mass theorem to the non-existence of psc 
metrics on manifolds of the form $T^n\# N$
where $N$ is some closed oriented manifold. We adapt this argument to our lower regularity 
setting. For other positive mass theorems in low-regularity settings, see the results 
of Lee-LeFloch \cite{LL}.

Let us describe our second main result, Theorem B.
When adapting the Li-Mantoulidis methods to the $4$-dimensional setting, 
one is faced with the oriented bordism group of 
psc $3$-manifolds. A pair of closed oriented psc $n$-manifolds
$(M_0,g_0)$, $(M_1,g_1)$ are said to be {\em psc-bordant} if there is a 
compact oriented $(n+1)$-manifold $W$
equipped with a psc metric $\bar{g}$ such that
\begin{enumerate}
\item $W$ is an oriented cobordism from $M_0$ to $M_1$, i.e. $\partial W=M_0\sqcup-M_1$;
\item The metric $\bar{g}$ takes the product form $g_i+dt^2$ on a neighborhood of $M_i$ for $i=0,1$.
\end{enumerate}
Considering the class of all closed oriented psc $n$-manifolds modulo psc-bordism with the operation 
of disjoint union, one obtains an abelian group $\Omega^{SO,+}_n$ for $n\geq2$. One can also
consider psc manifolds equipped with the extra data of a map to a fixed space $X$ 
and additionally require that
cobordisms $(W,\bar{g})$ admit extensions of this map. This yields an abelian group 
$\Omega^{SO,+}_n(X)$. 

Very little is known about the groups $\Omega^{SO,+}_n$.
In fact, the only previously computed psc-bordism group is 
$\Omega^{SO,+}_2$, which is trivial. One could argue that, 
though these groups were not defined at the time, the triviality of 
$\Omega^{SO,+}_2$ dates back to Weyl in 
1916 where he showed the path-connectedness of the space
of psc metrics on $S^2$. 
We mention a related result: in \cite{MS}, 
Mantoulidis-Schoen show the connectedness
of the space of metrics on $S^2$ which are conformally psc.

Theorem B, which may be of independent interest, 
is a calculation of the groups 
$\Omega^{SO,+}_3(S)$ where $S$ is a finite $1$-complex, i.e. $S$ is a union of points and wedge-products of circles. 
It is an essential tool in our proof of Theorem A. 
\begin{theoremB}
Let $S$ be a finite $1$-complex. Then the $3$-dimensional oriented psc-bordism group 
$\Omega^{SO,+}_3(S)$ is trivial.
\end{theoremB}

Let us provide this result with some context. It has been known since the 1950's, 
see Rohlin \cite{Roh} and Thom \cite{Thom}, that every oriented $3$-manifold 
is the boundary of
some compact oriented $4$-manifold, which we refer to as a {\em{null-cobordism}}. 
There are many proofs of this fact, both constructive
and abstract, but it is far from clear whether or not null-cobordisms of 
psc $3$-manifolds may 
themselves be equipped with
psc metrics. However, for spherical space forms $S^3/\Gamma$ where $\Gamma\subset SO(4)$
acts freely on $S^3$, algebraic geometry provides promising null-cobordisms for our purposes.
Indeed, if $\Gamma\subset U(2)$, the minimal resolution of the flat cone $\mathbb{C}^2/\Gamma$, 
\begin{equation}
X\to\mathbb{C}^2/\Gamma
\end{equation}
often admits special geometries and large compact sets in $X$ are null-cobordisms of $S^3/\Gamma$.
We adopt this strategy and rely on the remarkable constructions of Lock-Viaclovsky \cite{LV19}. 
Another ingredient is the connectedness of
the moduli space of psc metrics on $3$-manifolds due to Marques \cite{Marques}.
We remark that a generalization of this connectedness result has recently been obtained by Carlotto-Li \cite{CL}.

\subsection{Structure of the paper and further discussion}\label{s:disc}
We begin in Section \ref{s:back} by recalling the notions and tools we will 
require throughout the paper.
These include conformal geometry on manifolds with boundary, the 
structure of $3$-manifolds admitting 
psc metrics, and the asymptotically locally Euclidean (ALE) metrics we use. 
The scalar-flat K{\"a}hler ALE metrics constructed by Lock-Viaclovky 
\cite{LV19} are of fundamental importance to our work --
large regions in these manifolds are the building blocks for the psc 
null-cobordisms we build to prove Theorem B in Section \ref{s:theoremB}.
Theorem A is proven in Section \ref{s:main}, though we relegate
a technical discussion to the Appendix. 
We conclude by establishing Corollaries A and B in Section \ref{s:cors}.

Now let us outline the proof of Theorem A and discuss the status of Conjecture 
\ref{con:schoen} in dimension $4$.
In Section \ref{s:main}, we prove Theorem A by following a strategy developed in \cite{LM}
and draw heavily from the analytical results there. Given a singular psc $4$-manifold $(M,g)$, 
we geometrically ``blow-up'' the tolerable singular set $S$ in such a way that a neck 
appears near each component of $S$.
In these necks, one finds minimal hypersurfaces. In analogy to the horizon of a black hole in
General Relativity, these hypersurfaces act as shields, protecting the regular part of the manifold
from the singular set. A classical argument of Schoen-Yau \cite{SY79} shows that these shields
admit psc metrics. Next, we excise the singular regions behind the shields and 
cap-off the resulting boundary with 
the null psc-cobordisms we construct in Theorem B, 
producing a {\em smooth} 
closed psc manifold $(\overline{M},\overline{g})$. 
Due to the nature of our psc null-cobordisms, $\overline{M}$ will generally not be 
diffeomorphic to $M$, having much larger second Betti number. 

Let us describe this point in more detail. The building-block ALE manifolds we use have non-trivial
second homology. Moreover, by considering the index of the Dirac operator, 
{\em{any}} (spin) null psc-cobordism of a non-trivial space form $S^3/\Gamma$
must have non-zero signature, see \cite{Dahl97} for a nice description of this phenomenon. 
To consider an example, it may be possible that one of the 
shielding hypersurfaces is the connected sum of a lens space with itself --
see \cite{Donald} for work on the underlying embeddability question. 
In our desingularization procedure, we cap-off this hypersurface by taking the boundary connected sum of two 
copies of the null psc-cobordism of the lens space,
introducing much new topology. That being said, we cannot rule out the existence of 
a more simple psc cap which would not increase topology. It is possible, 
therefore, that the change in topology is merely a defect of 
our methods.

Nevertheless, our desingularization $\overline{M}$ will admit a degree-$1$ 
map back to $M$ which collapses the added topology to the singular set. 
Because some components of $S$ may represent non-trivial elements of $\pi_1(M)$,
care must be taken to arrange for this degree-$1$ map to exist. Indeed, 
if one is not deliberate, the process of
excising and capping-off the singular set may have the effect of completely 
killing the fundamental group, potentially destroying
the possibility of such a degree-$1$ map. It is for this reason we require 
the triviality of $\Omega^{SO,+}_3(S)$ and 
not just the smaller group $\Omega^{SO,+}_3$.

We conclude this section with some informal remarks on 
Conjecture \ref{con:schoen}. 
In Theorem A we consider the
class of enlargeable manifolds because the desingularization of 
$M$ always has a degree-$1$ map back to $M$
and thus does not leave the class of enlargeable manifolds. However, 
there are many $4$-manifolds which are not enlargeable
yet do not admit psc metrics: certain manifolds with non-trivial Seiberg-Witten 
invariant. Unlike enlargeable manifolds, this class is not 
closed under non-zero degree maps and so Theorem A has no immediate
consequence for singular psc metrics on manifolds of this type. For instance, 
there is a homeomorphism
\[
\#^3\mathbb{C}P^2\#^{20}\overline{\mathbb{C}P^2}\to 
\mathrm{K}3\#\overline{\mathbb{C}P^2}.
\]
However, the one point blow up of K$3$ has non-trivial Seiberg-Witten 
invariant whereas the Gromov-Lawson 
construction \cite{GL80} shows that the connected sum of $\mathbb{C}P^2$'s 
admits psc metrics.
For this reason, we expect that either a much more refined version of our technique or
entirely new ideas are needed to approach Conjecture \ref{con:schoen} for general
$4$-manifolds.

The situation appears worse in higher dimensions since almost nothing
is known about the groups $\Omega^{SO,+}_n$ for $n\geq4$. 
For instance, in our proof of
Theorem B, we make full use of the classification of psc $3$-manifolds
and the connectedness of the moduli space of psc metrics on a fixed $3$-manifold.
Starting in dimension $4$, not only is there no classification of psc manifolds, but
the moduli space of psc metrics is known to be disconnected in general, see
\cite{Ruberman} for dimension $4$ and \cite{BERW} for higher dimensions.

\subsection{Acknowledgments}
I owe a debt of gratitude to Bradley Burdick for many 
enlightening conversations.
I would also like to thank Professor Michael Anderson, 
Professor Christina Sormani, Professor Marcus Khuri, and Professor Blaine Lawson
for their interest and helpful suggestions. 
For other helpful conversations, I thank Chao Li, Christos Mantoulidis, 
and Michael Albanese. The anonymous referee provided invaluable and helpful comments.

\section{Background}\label{s:back}

\subsection{Enlargeable manifolds}\label{s:en}
In this subsection, we will recall the enlargeability condition appearing in Theorem A and its
fundamental properties. The notion of enlargeability was originally introduced by
Gromov-Lawson in \cite{GL80} and substantially generalized in \cite{GL83}, which we refer
readers to for a general discussion.

We follow the language of \cite{GL83}, but omit technical considerations 
related to non-compactness and spin structures since we will not require 
them. For $\varepsilon>0$, a compact Riemannian $n$-dimensional manifold $(M,g)$ is 
said to be {\emph{$\varepsilon$-hyperspherical}} if there exists a map $f:M\to S^n$ satisfying
\begin{equation}
||df(V)||_{g_{rnd}}\leq\varepsilon ||V||_g
\end{equation}
for all non-zero vectors $V\in TM$ where $g_{rnd}$ denotes the round metric of radius $1$.
\begin{definition}\label{d:en}
A compact manifold $M$ is called {\emph{enlargeable}} if for all Riemannian metrics $g$ and $\varepsilon>0$, there exists
a cover $\overline{M}\to M$ which, equipping $\overline{M}$ with the lift of $g$, is $\varepsilon$-hyperspherical.
\end{definition}
The prototypical examples of enlargeable manifolds are compact manifolds admitting non-positive sectional curvature metrics and 
compact solvmanifolds such as the torus $T^n$. Enlargeability is
an invariant of the homotopy-type of the manifold. Even more, the class of enlargeable manifolds is closed under 
non-zero degree maps.
\begin{prop}\cite{GL80}\label{p:en}
Suppose $M^n$ is a closed enlargeable $n$-manifold.
If $N^n$ is a closed $n$-manifold and there exists a map $N\to M$ of non-zero degree,
then $N$ is also enlargeable.
\end{prop}
\noindent A basic application of Proposition \ref{p:en}: the connected sum $T^n\# M$ of a torus 
with a closed manifold $M$ is enlargeable by considering the map which collapses $M$ to a point.

By considering a twisted Dirac operator, Gromov-Lawson made the fundamental 
observation that enlargeable spin manifolds 
cannot admit psc metrics \cite[Theorem A]{GL83}. Recently,
Cecchini-Schick have generalized this to the non-spin case.
\begin{theorem}\cite[Theorem A]{CS}\label{t:en}
A closed enlargeable $n$-dimensional manifold cannot admit a psc metric.
\end{theorem}
The proof of Theorem \ref{t:en} involves a clever construction related
to Gromov's torical bands \cite{Gro18} and 
involving the notion of {\emph{minimal $k$-slicings}}
introduced by Schoen-Yau in the preprint \cite{SY17}. We remark that our applications of Theorem \ref{t:en}
are in dimension $4$ and therefore do not require the full regularity theory developed in \cite{SY17}.

\subsection{The conformal Laplacian with minimal boundary conditions}
In later sections we will require some basic facts about
the Yamabe problem on compact manifolds. For brevity, we will treat the cases in which 
the manifold has empty and non-empty boundary simultaneously.

Now let $(W,\bar h)$
be an $n$-dimensional manifold with possibly non-empty boundary $(\p W, h)$ 
where $h=\bar h|_{\p W}$. We consider the following pair of operators acting on
$C^\infty(W)$:
\begin{equation}\label{e:boundaryconf}
\left\{
\begin{array}{lcll}
  \mathcal{L}_{\bar h} &=& -\Delta_{\bar h}+c_n R_{\bar h} 
 & \mbox{in $W$}
  \\
\mathcal{B}_{\bar h} & = & \p_\nu+2c_nH_{\bar h}  & \mbox{on $\p W$},
\end{array}
\right.
\end{equation}
where $c_n=\frac{4(n-1)}{n-2}$ and $H_{\bar{h}}$ is the mean curvature
of $\p W$ with respect to $\nu$, the outward-pointing normal vector.

Recall that if $\phi\in C^\infty(W)$ is a positive
function, then the scalar and boundary mean curvatures of the
conformal metric $\tilde h=\phi^{\frac{4}{n-2}}\bar h$ are given by
\begin{equation}\label{eq:conformal}
\left\{
\begin{array}{lcll}
  R_{\tilde{h}} &=& c_n^{-1} \phi^{-\frac{n+2}{n-2}} \cdot \mathcal{L}_{\bar
    h}\phi & \mbox{in $W$} \\ H_{\tilde{h}} & = &
  \frac{1}{2}c_n^{-1}\phi^{-\frac{n}{n-2}} \cdot \mathcal{B}_{\bar h}\phi &
  \mbox{on $\p W$}.
\end{array}
\right.
\end{equation}
We consider the Rayleigh quotient coming from (\ref{e:boundaryconf}) and take the infimum:
\begin{equation}\label{eq:Rayleigh}
\lambda_1=\inf_{\phi\not\equiv0\in H^1(W)}
\frac{\int_W\left(|\nabla\phi|^2+c_nR_{\bar h}\phi^2\right)d\mu+2c_n\int_{\p W}H_{\bar h}\phi^2d\sigma}{\int_W\phi^2d\mu}.
\end{equation}
According to standard elliptic PDE theory, we obtain an elliptic
boundary problem, denoted by $(\mathcal{L}_{\bar h}, \mathcal{B}_{\bar h})$,
and $\lambda_1$ is
the principal eigenvalue of the minimal boundary problem
$(\mathcal{L}_{\bar h}, \mathcal{B}_{\bar h})$. The corresponding Euler-Lagrange
equations are the following:
\begin{equation}
\label{eq:eigenvalue}
\left\{
\begin{array}{lcll}
  \mathcal{L}_{\bar h}\phi &= & \lambda_1\phi& \text{ in }W
  \\
  \mathcal{B}_{\bar h}\phi &= & 0 &\text{ on } \p W.
\end{array}
\right.
\end{equation}  
This problem was first studied by Escobar \cite{E1} in the context of
the Yamabe problem on manifolds with boundary.  

Let $\phi$ be a solution of (\ref{eq:eigenvalue}).  It is
well-known that the eigenfunction $\phi$ is smooth and can be chosen to
be positive.  
A straight-forward computation shows that the conformal metric $\tilde h=
\phi^{\frac{4}{n-2}}\bar h$ has the following scalar and mean
curvatures:
\begin{equation}\label{eq:6}
  \left\{
\begin{array}{lcll}
  R_{\tilde  h} &= & \lambda_1 \phi^{-\frac{4}{n-2}} & \mbox{in $W$}
\\
  H_{\tilde  h} &\equiv &  0 & \mbox{on $\p W $}.
\end{array}
\right. 
\end{equation} 
In particular, the sign of the eigenvalue $\lambda_1$ is a 
conformal invariant, see \cite{E1,E3}. We call a conformal manifold
$(W,[h])$ Yamabe positive, null, or negative according to the sign of $\lambda_1$.

The link between the above conformal considerations and psc bordisms
is given by Akutagawa-Botvinnik in \cite{AB}. We will
use the following approximation trick throughout the paper.
\begin{theorem}\cite[Corollary B]{AB}\label{t:AB}
Let $W^n$ be a manifold of dimension $n\geq3$ with non-empty boundary $\p W$.
Let $g$ be a psc metric on $\p W$. The following are equivalent:
\begin{enumerate}
\item There is a metric $\bar{g}$ on $W$ so that $g\in[\bar{g}|_{\p W}$ and $(W,[\bar{g}])$
is Yamabe positive;
\item There is a psc metric $\tilde{g}$ on $W$ which takes the product form $\tilde{g}=g+dt^2$
on a neighborhood of $\p W$.
\end{enumerate}
\end{theorem}

\subsection{Psc $3$-manifolds}
In order to prove Theorem B, we will need the classification 
of closed oriented $3$-manifolds admitting psc metrics.
More precisely, the following result is
a consequence of Gromov-Lawson \cite[Theorem 8.1]{GL83},
Schoen-Yau \cite{SY79b},
and Perelman's solution of the geometrization conjecture 
\cite{per1,per2,per3}.
\begin{prop}\label{p:classification}
Suppose $M$ is a closed oriented $3$-manifold. Then 
$M$ admits a psc metric if and only if it is 
diffeomorphic to the 
connected sum $\#_{i=1}^kM_i$ where 
each $M_i$ is one of the following:
\begin{enumerate}
\item $S^2\times S^1$;
\item $S^3/\Gamma$ where $\Gamma\subset SO(4)$ is a finite subgroup acting 
freely on $S^3$.
\end{enumerate}
\end{prop}

Another ingredient for the proof of Theorem B is
the work of Marques in \cite{Marques} on the space of psc metrics on $3$-manifolds.
To state it, for a manifold $M$, let $\mathrm{Riem}^+(M)$ denote the space
of psc metrics on $M$, endowed with the Whitney topology. The group of
diffeomorphisms of $M$, denoted $\mathrm{Diff}(M)$, acts on $\mathrm{Riem}^+(M)$ and
one may form the {\emph{moduli space of psc metrics}}
\begin{equation}
\mathcal{M}_+(M)=\mathrm{Riem}^+(M)/\mathrm{Diff}(M).
\end{equation}
By making careful use of Perelman's Ricci flow with surgery and the Gromov-Lawson
surgery construction, Marques obtained the following connectedness result.
\begin{theorem}\cite[Main Theorem]{Marques}\label{t:connectedness}
Suppose $M$ is a closed oriented $3$-manifold which admits a psc metric.
Then the moduli space $\mathcal{M}_+(M)$ is path connected.
\end{theorem}

\subsection{ALE manifolds and mass}\label{s:ALEpre}
The final ingredient we we will require is a family of asymptotically locally Euclidean (ALE) 
metrics. We refer the reader to \cite{Bart} for a complete discussion of such spaces
and their fundamental analytic properties. 
\begin{definition}\label{d:ALE}
Let $\Gamma\subset SO(n)$ be a finite subgroup acting freely on $S^{n-1}$ and let $\tau>0$.
A complete $n$-dimensional Riemannian manifold $(X,g)$ is said to be 
{\emph{ALE of order $\tau$ with group at infinity $\Gamma$}}
if
\begin{enumerate}
\item There is a compact subset $K\subset X$ and a diffeomorphism called a {\em{chart at infinity}}
\[
\Psi:X\setminus K\to (\mathbb{R}^n\setminus B)/\Gamma
\]
where $B\subset \mathbb{R}^n$ is a ball centered at the origin;
\item Denoting the radial coordinate on $\mathbb{R}^n/\Gamma$ by $r$, we have
\[
|\partial^\alpha((\Psi^{-1})^*g_{ij}-\delta_{ij})|=\mathcal{O}(r^{-|\alpha|-\tau})
\]
for any multi-index $\alpha$ with $|\alpha|=0,1,2$ and $\delta$ is the Euclidean metric.
\end{enumerate}
\end{definition}
\noindent If an ALE manifold has trivial group at infinity, the manifold is said to be {\emph{asymptotically flat}}.

If one fixes a base point $p$ in an ALE manifold $X$, notice that large geodesic spheres $S(p,R)$
become asymptotic to a spherical space form $S^{n-1}/\Gamma$. This space form is independent of $p$
and we call $S^{n-1}/\Gamma$ the {\emph{sphere at infinity}} of $X$.

We will be interested in $4$-dimensional ALE manifolds with non-negative scalar curvature. In general, it is a difficult problem to find scalar non-negative ALE manifolds with a given group $\Gamma\subset SO(4)$ at infinity. Such constructions have a history dating back to \cite{Hawk} too rich to cover here, so we will be brief. In \cite{Kron1,Kron2}, Kronheimer produced hyper-K{\"a}hler (and hence Ricci-flat) ALE manifolds for any finite $\Gamma\subset SU(2)\subset SO(4)$ at infinity. The negative-mass metrics constructed by LeBrun in \cite{LeBrun} provide examples of scalar-flat K{\"a}hler ALE metrics having cyclic groups at infinity not contained $SU(2)$. Calderbank-Singer \cite{CSing} constructed Ricci-flat ALE metrics with any given $3$-dimensional lens space as the sphere at infinity.

However, it was not until recently that {\emph{every}} finite $\Gamma\subset U(2)\subset SO(4)$ acting freely on $S^3$ was shown to arise as the group at infinity of a scalar-flat ALE manifold. Through a delicate assembly, Lock-Viaclovsky were able to piece together the known constructions to obtain the following:
\begin{theorem}\cite[Theorem 1.3]{LV19}\label{t:smorg}
Let $\Gamma\subset U(2)$ be a finite group containing no complex reflections. Then there are scalar-flat K{\"a}hler ALE metrics on the minimal resolution of $\mathbb{C}^2/\Gamma$.
\end{theorem}
To explain the relevance to Theorem B, {\emph{any}} finite subgroup of $SO(4)$ acting freely on $S^3$ is conjugate within $SO(4)$ to a subgroup of either the standard embedding $U(2)\subset SO(4)$ or its so-called orientation-reversed embedding, see \cite[Theorem 4.10]{Scott}. It follows that the diffeomorphism type of any $3$-dimensional spherical space form can be realized as the quotient of $S^3$ by a subgroup of $U(2)$. Using this observation with Theorem \ref{t:smorg}, one immediately obtains the following.
\begin{cor}\label{cor:ALE}
Let $N^3$ be an orientable $3$-dimensional spherical space form. Then there is a scalar-flat ALE manifold whose sphere at infinity is diffeomorphic to $N$.
\end{cor}

We conclude this section by recalling the definition of the {\em{mass}} of an ALE manifold appearing in Corollary B.
An ALE $n$-manifold $(M,g)$ can be viewed as a time-symmetric slice of an $(n+1)$-dimensional space-time
and one may consider its ADM mass
\begin{equation}
m(M,g)=\lim_{R\to\infty}\frac{\Gamma(\frac{n}{2})}{4(n-1)\pi^{n/2}}\int_{S_R/\Gamma}\sum_{i,j=1}^n(g_{ij,i}-g_{ii,j})\nu^j d\vol_\delta
\end{equation}
where $S_R$ is the Euclidean coordinate sphere of radius $R$, the indicies arise from coordinates given by the chart at infinity, $\nu$
is the outward-pointing Euclidean normal vector to $S_R/\Gamma\subset\mathbb{R}^n/\Gamma$, and $\Gamma(\frac{n}{2})$ is a value of the Gamma function (unrelated to the group $\Gamma$).
If the order of decay $\tau$ satisfies $\tau>(n-2)/2$ and $R_g\in L^1(M)$, then $m(M,g)$ is independent of
the chart at infinity, due to \cite{Bart} and \cite{Chru}.

If $(M,g)$ is asymptotically flat, of dimension less than $8$, and has non-negative scalar curvature, then the 
celebrated positive mass theorem 
of Schoen-Yau \cite{SY79c} states that $m(M,g)\geq0$ and that the 
only such manifold with vanishing mass is Euclidean space. For the higher dimensional case,
see \cite{SY17} and \cite{Loh16}. 

\section{Proof of Theorem B}\label{s:theoremB}
This section is devoted to the proof of Theorem B. 
Before beginning,
we make some preparations.
The following is an elementary observation about large sets in
ALE manifolds of any dimension.
\begin{lemma}\label{l:ALEconversion}
Let $(X,g)$ be an $n$-dimensional ALE manifold with group at infinity $\Gamma\subset SO(n)$
and let $p\in X$. Given $\varepsilon>0$, there exists a bounded open set $\Omega\subset X$ such that
\begin{enumerate}
\item $\partial\Omega$ is diffeomorphic to $S^{n-1}/\Gamma$;
\item the restriction metric $g|_{\p\Omega}$ is $\varepsilon$-close to a metric of constant positive 
sectional curvature on $S^{n-1}/\Gamma$ in the $C^2$-topology;
\item $\p\Omega$ has positive mean curvature with respect to the normal vector pointing towards infinity.
\end{enumerate}
\end{lemma}
\begin{proof}
Choose a compact set $K$ and diffeomorphism $\Phi$ as in Definition \ref{d:ALE}.
For $R>>0$ let $B_R$ denote the coordinate ball in $(\mathbb{R}\setminus B)/\Gamma$.
Considering the decay of $g$, $\frac{1}{R^2}(\Phi^{-1})^*g$ converges to 
the flat metric on the cone $\mathbb{R}^n/\Gamma$ as $R\to\infty$ in the $C^2$-topology.
It follows that, for sufficiently large $R$, $\Phi^{-1}(\partial B_R)$ has positive mean curvature and the
restriction of $g$ to $\Phi^{-1}(\partial B_R)$ is
$\varepsilon$-close to the metric of constant curvature $\frac{1}{R}$. The set $\Omega=K\cup\Phi^{-1}(B_R)$
is the desired open set.
\end{proof}

\noindent The following is a simple consequence of the connectedness result
Theorem \ref{t:connectedness}.
\begin{prop}\label{p:equiv}
Fix a closed oriented $3$-dimensional manifold $M^3$ which admits a psc metric. The following are equivalent
\begin{enumerate}
\item For any psc metric $g$ on $M$, $(M,g)$ is psc null-bordant.
\item There exists some psc metric $g_0$ on $M$ so that $(M,g_0)$ is psc null-bordant.
\end{enumerate}
\end{prop}
\begin{proof}
To prove the non-trivial implication $(2)\Rightarrow(1)$, let $g_0$ be a psc metric on $M$ 
and suppose there exists a compact $4$-dimensional psc manifold $(W^4,\bar g)$ so that $\partial W=M$
and $\bar g=g_0+dt^2$ near $M$. Let $g$ be a second psc metric on $M$.

By Theorem \ref{t:connectedness}, there is a path of psc metrics $\{g(t)\}_{t\in[0,1]}$ on $M$ 
so that $g(1)=g$ and $g(0)=\Phi^*g_0$ for some diffeomorphism $\Phi$ of $M$. By sufficiently slowing down the
parameterization of this path, one can obtain a psc-bordism of the form
\[
(M\times[0,L],g(s)+ds^2)
\]
between $(M,\Phi^*g_0)$ and $(M,g)$, see \cite{Rose}. Using $\Phi$, one may glue this 
cylinder to the boundary of $W$
to obtain
\begin{equation}
\left(W\bigcup_\Phi(M\times[0,L]), \bar g\cup(g'(s)+ds^2)\right),
\end{equation}
which is the promised psc null-cobordism of $(M,g)$.
\end{proof}

\vspace{.1in}

\noindent{\bf{Proof of Theorem B:}} First, we will show that $\Omega^{SO,+}_3$ is trivial. Let $M$ be a closed oriented $3$-manifold which admits a psc metric.
We first make some reductions.
By Proposition \ref{p:equiv}, it suffices to construct a psc null-cobordism of $(M,g_0)$
for a single psc metric $g_0$ on $M$. In order to achieve this, by Theorem \ref{t:AB}, it is
sufficient to find a Yamabe positive conformal manifold $(W,[\bar g])$ whose boundary $(M,[\bar{g}|_{M}])$ 
is also Yamabe positive.

We will first consider the case where $M$ is one of the irreducible manifolds appearing in Proposition \ref{p:classification}. 
For $M= S^2\times S^1$, consider the manifold $W=S^3_+\times S^1$
where $S^3_+$ denotes the upper hemisphere of the round $3$-sphere. With the product metric $\bar{g}$,
$S^3_+\times S^1$ is psc with minimal boundary isometric to the product of the round $S^2$ and $S^1$.
By inspecting the Rayleigh quotient (\ref{eq:Rayleigh}), $(S^3_+\times S^1,[\bar{g}])$ is 
Yamabe positive, competing our work in this case. 

Now consider the case of a spherical space form $M= S^3/\Gamma$.
By Corollary \ref{cor:ALE} there is a scalar-flat ALE manifold $(X,g_X)$ with sphere at infinity 
$S^3/\Gamma$. Applying Lemma \ref{l:ALEconversion} with some $\varepsilon>0$, 
we obtain a large compact set $K_\varepsilon\subset X$ with metric $\bar{g}_\varepsilon=g_X|_{K_\varepsilon}$
which is scalar-flat with mean convex boundary $S^3/\Gamma$. By perhaps choosing smaller
$\varepsilon$, we can ensure that the restriction metric $\bar{g}_\varepsilon|_{S^3/\Gamma}$
is sufficiently close to a round metric to have positive scalar curvature. Due to the positive mean curvature
of its boundary, $(K_\varepsilon,[\bar{g}_\varepsilon])$ is Yamabe positive, finishing our 
work with the space forms.

Finally, we consider general $M$. By Proposition \ref{p:classification}, $M$ may be decomposed as
$M=\#_{i=1}^k M_i$ where each $M_i$ is with $S^2\times S^1$ or a space form $S^3/\Gamma$.
By our work above, we may consider $(W_i,\bar{g}_i)$ which are psc null-cobordisms of some psc metrics
$g_i$ on $M_i$ for each $i=1,\dots,k$. Notice that the boundary connected sum
\[
W=\bignatural_{i=1}^k W_i
\]
is a null-cobordism of $M$. To make $W$ into a psc null-cobordism, we use a version of
the classical Gromov-Lawson surgery construction \cite{GL80b} adapted to the case of
boundary connected sums by Carr \cite[Lemma 10]{Carr}, c.f. \cite{Walsh}. 
This produces a psc metric on $W$ which is 
product near the boundary and resembles
a Gromov-Lawson connected sum construction 
applied to the $3$-manifolds 
$(M_i,\bar{g}_i|_{M_i})$ near the boundary of $W$. 
It follows that there is a psc null-bordant
psc metric on $M$. This shows that $\Omega^{SO,+}_3$ is
trivial. 

In Figure 2 we depict the null-cobordism
of a general psc $(M^3,g)$ we have obtained by combining
all the constructions in this section.

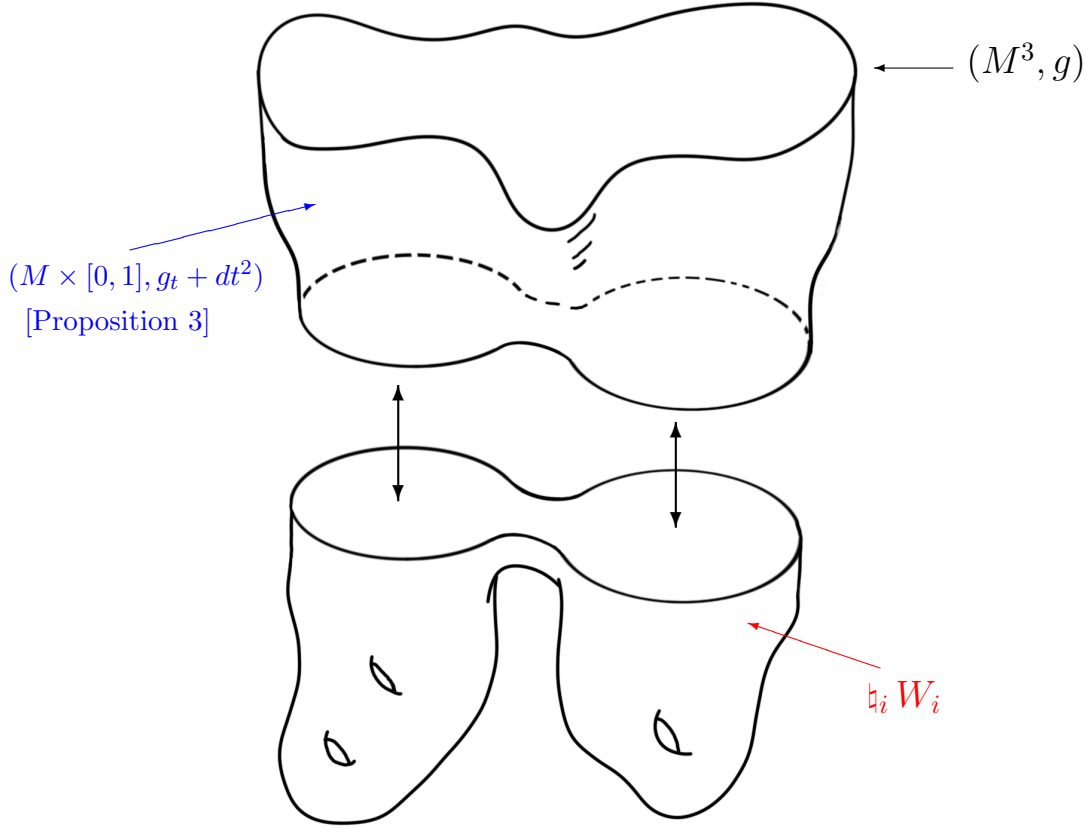
\begin{figure}[!htb]
\label{f:cob}
\includegraphics[scale=.3]{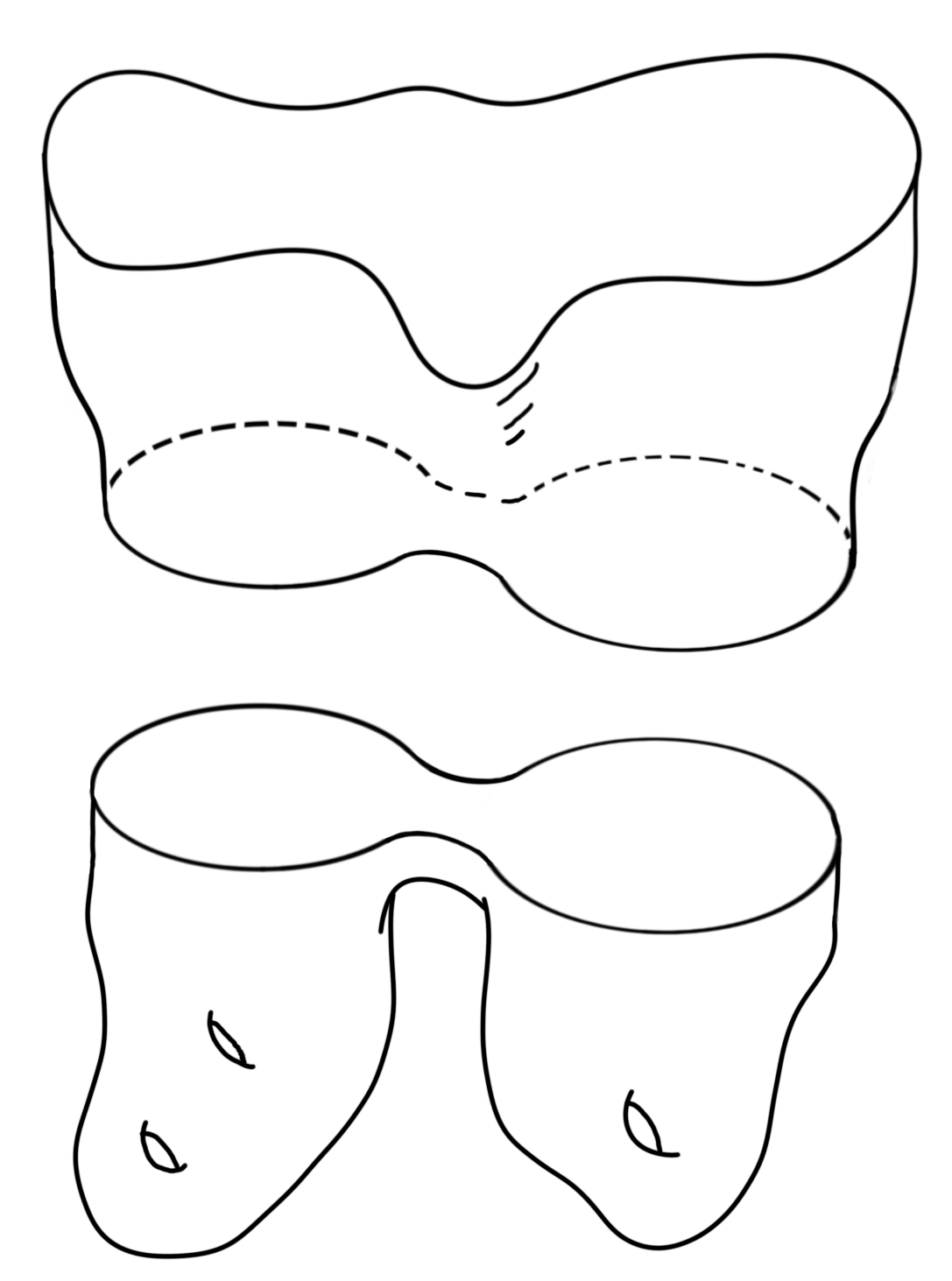}
\begin{picture}(0,0)
\put(30,290){{\Large $(M^3,g)$}}
\put(25,292){\vector(-1,0){30}}
{\thicklines
\put(-80,118){\vector(0,1){40}}
\put(-80,118){\vector(0,-1){0}}
\put(-185,128){\vector(0,1){44}}
\put(-185,128){\vector(0,-1){0}}
}
{\color{blue}
\put(-340,210){{ $(M\times[0,1],g_t+dt^2)$}}
\put(-290,223){\vector(4,1){70}}
\put(-330,193){[Proposition \ref{p:equiv}]}
}
{\color{red}
\put(-15,50){{\Large $\bignatural_i W_i$}}
\put(-10,65){\vector(-3,1){50}}
}
\end{picture}
\caption{The null-cobordism we construct of a psc $(M^3,g)$}
\end{figure}

Now we turn our attention to $\Omega^{SO,+}_3(S)$ where $S$ is some finite $1$-complex.
Let $(M,g)$ be a closed oriented psc $3$-manifold equipped with a map $F:M\to S$.
We will assume that $M$ and $S$ are connected -- if they are not, 
one may apply the proceeding 
argument to each component. By Proposition 
\ref{p:classification}, we have the decomposition
\[
M\cong(\#_{i=1}^{k_1}S^3/\Gamma_i)\#(\#^{k_2}(S^2\times S^1))
\]
where we allow $k_1$ and $k_2$ to possibly take the value of $0$.
Evidently, the bordism class $[(M,g,F)]$ is unchanged by homotopies of $F$.
As such, it suffices to consider the case where $S$ is the wedge of $l$ circles $\vee^lS^1$.

An exercise in basic algebraic topology shows there is a bijection
\[
[M,\vee^lS^1]\to \mathrm{Hom}(\pi_1(M),F_l)\cong\mathrm{Hom}(\Gamma_1*\cdots*\Gamma_{k_1}*F_{k_2},F_l)
\]
where $[M,\vee^lS^1]$ is the set of homotopy classes of maps from $M$ to $\vee^lS^1$, 
$F_l$ is the free group on $l$ generators, and $*$ is the free product.
Since all $\Gamma_i$ are finite and $F_l$ has no torsion elements, it follows that $F$ is homotopic 
to a map $F':M\to \vee^lS^1$ which maps the summand 
$(\#_{i=1}^{k_1}S^3/\Gamma_i)$ to a single point $x_0\in \vee^lS^1$.

Now consider the psc null-cobordism of $(M,g)$ we constructed above:
\[
(W,\bar{g})=\left(\left(\bignatural_{i=1}^{k_1} W_i\right)\bignatural\left(\bignatural^{k_2}S^3_+\times S^1\right),\bar{g}\right).
\]
We proceed by constructing an extension of $F'$ to a map $\overline{F}:W\to S^1$ in two steps.
First, declare $\overline{F}|_{\p W}=F'$ and that $\overline{F}$ map the summand $\natural_{i=1}^{k_1} W_i$ to 
the single point $x_0$. 
Since $\overline{F}$ is constant on the $\natural_{i=1}^{k_1} W_i$ summand, it suffices to extend $\overline{F}$ over
\[
W\setminus \left(\bignatural_{i=1}^{k_1} W_i\right)\cong \bignatural^{k_2}S^3_+\times S^1.
\]
Since the only $2$-cells in $\natural^{k_2}S^3_+\times S^1$ lie on its boundary where $\overline{F}$ is already defined,
all classes obstructing the extension of $\overline{F}$ vanish. For readers unfamiliar with obstruction theory, see \cite[Section 4.3]{Hatcher}. 
It follows that $\overline{F}$ may be extended over all 
of $W$ and so $(W,\bar{g},\overline{F})$
is our desired null-cobordism. This finishes the proof of Theorem B.

\section{Proof of Theorem A}\label{s:main}
Now that we have established the required bordism calculations, 
we are ready to prove Theorem A. 

\noindent {\bf{Proof of Theorem A:}} 
Let $M$, $S$, and $g$ be as in Theorem A. According to the condition $g\in L^\infty_E(M)\cap C^{\infty}(M\setminus S)$, there is a smooth metric $h$ and a positive constant $C$ so that $g$ is bounded above and below by $Ch$ and $h$, respectively.
Fix some positive function $\underline{R}$ on $M$ such that
\begin{equation}
0< \underline{R}\leq \min(1,R_g)
\end{equation}
which will act as a bounded surrogate for the possibly unbounded scalar curvature of $g$.
We may consider 
a positive Green's-type function, $G$, of the modified conformal Laplacian 
\begin{equation}\label{e:almostconf}
\underline{\mathcal{L}_g}=-\Delta_g+\frac83 \underline{R},
\end{equation}
see \cite[Theorem 6.1]{LSW}.
As a distribution, $G$ solves the equation
\begin{equation}
\left(-\Delta_g+\frac83 \underline{R}\right)G=\delta_S
\end{equation}
where $\delta_S$ is the Dirac measure of $S$. Strictly speaking, \cite{LSW} considers only Green's functions with singularities along
points. However, generalizing this construction to find
functions in the kernel of positive linear elliptic operators 
with singularities along smoothly 
embedded submanifolds of codimension greater than $2$ is classical,
see \cite{SY79}.

Due to the uniform ellipticity of $\underline{\mathcal{L}_g}$, the asymptotics of $G$ near $S$ behave in
a standard manner, see \cite[Theorem 7.1]{LSW} and \cite[Appendix]{SY79}. 
More precisely, there is a constant $C_0>0$ so that:
near an isolated point singularity $x_0\in S$,
\begin{equation}\label{e:greens1}
C_0^{-1}d_h(x_0,x)^{-2}\leq G(x)\leq C_0d_h(x_0,x)^{-2}
\end{equation}
and near a $1$-dimensional component $L\subset S$
\begin{equation}\label{e:greens2}
C_0^{-1}d_h(L,x)^{-1}\leq G(x)\leq C_0d_h(L,x)^{-1}.
\end{equation}
Next, we study the family of metrics
\begin{equation}
g_\varepsilon=(1+\varepsilon G)^2g
\end{equation}
on $M\setminus S$ where $\varepsilon$ is a small positive parameter. A straight forward calculation using (\ref{e:almostconf}) shows $\mathcal{L}_g(1+\varepsilon G)>0$. Inspecting the conformal transformation formula (\ref{eq:conformal}), it follows that $g_\varepsilon$ has positive scalar curvature.

Given a small distance $\eta>0$, consider the tubular neighborhood $U^\eta=\{x\in M\colon d_g(x,S)<\eta\}$ and fix $\eta_0>0$ small enough so that there is a map $f:U^{\eta_0}\to S$ 
deformation retracting $U^{\eta_0}$ onto $S$. Next, we consider the non-compact 
manifold $Q=\overline{U^{\eta_0}}\setminus S$ with the restriction metric, which we continue denoting by $g_\varepsilon$.

\begin{claim}\label{c:horizon}
There is a $\varepsilon_0>0$ so that, for each $\varepsilon\leq\varepsilon_0$, there is 
a closed smoothly embedded hypersurface $\Sigma_{\varepsilon}\subset U^{\eta_0}$ satisfying
\begin{enumerate}
\item $\Sigma_{\varepsilon}$ is stable-minimal with respect to $g_\varepsilon$;
\item $\overline{U^{\eta_0}}\setminus\Sigma_{\varepsilon}$ consists of two components:
$A_{\varepsilon}^{in}$ containing $S$ and $A_{\varepsilon}^{out}$ containing $\partial U^{\eta_0}$.
\end{enumerate}
\end{claim}
\begin{proof}
For each $\varepsilon>0$, consider the minimization problem
\begin{equation}\label{e:min}
V_{\varepsilon}=\inf\{\vol_{g_\varepsilon}(\p\Omega)\colon \Omega\subset Q\text{ is open and contains }
	\p U^{\eta_0}\}.
\end{equation}
The first step is to leverage Lemma \ref{l:minsol} in the Appendix to study a solution of this minimization problem. In order to explain how Lemma \ref{l:minsol} applies, we must take a moment to articulate \eqref{e:min} in a more amenable form. 

By repeating the argument for each path component of $S$, it suffices to consider the case where $S$ is connected. The case where $S$ is a point is treated in \cite{LM}, so we will assume $S$ is an embedded circle. Choosing a potentially smaller $\eta_0$, one may consider Fermi coordinates $(t,y)$ on $U^{\eta_0}$ where $t$ parameterizes $S$ and $y$ ranges over the fibers of $S$'s normal bundle. Introduce polar coordinates $(\rho,\theta,\varphi)$ in the $y$ variables and consider a new radial coordinate $r=-\log\rho\in(-\log(\eta_0),\infty)$. Using the asymptotics (\ref{e:greens1}) and (\ref{e:greens2}), there is a constant $C_1$ so that $g_\varepsilon$ satisfies $C_1^{-1}\widehat{h}_\varepsilon\leq g_\varepsilon\leq C_1 \widehat{h}_\varepsilon$ over $U^{\eta_0}$ where $\widehat{h}_\varepsilon$ is given by 
\begin{align*}
\widehat{h}_\varepsilon&:=(1+\varepsilon /\rho)^2(dt^2+d\rho^2+\rho^2g_{S^2})\\
{}&=(1+\varepsilon e^{r})^2(dt^2+e^{-2r}dr^2+e^{-2r}g_{S^2})\\
{}&=(e^{-r}+\varepsilon)^2(dr^2+e^{2r}dt^2+g_{S^2}).
\end{align*}
At this point we have compared the metrics $g_\varepsilon$ and $\widehat{h}_\varepsilon$ on the domain $Q\cong [-\log\eta_0,\infty)\times S\times S^2\cong \left(\mathbb{R}^2\setminus B^2_{\eta_0}(0)\right)\times S^2$. Smoothly extend the restriction $g_\varepsilon|_{U^{\eta_0}}$ over $\mathbb{R}^2\times S^2$ to a metric $g'_\varepsilon$. For each $\varepsilon$, there is a constant $C_\varepsilon$ so that 
\begin{equation}\label{e:gepsilonbound}
C_\varepsilon^{-1}(g_{\mathbb{H}^2}+g_{S^2})\leq g_\varepsilon'\leq C_\varepsilon(g_{\mathbb{H}^2}+g_{S^2})
\end{equation}
on $\mathbb{R}^2\times S^2$ where $g_{\mathbb{H}^2}$ denotes the metric on the hyperbolic plane with sectional curvature $-1$. To be clear, the constant $C_\varepsilon$ depends heavily on $\varepsilon$ and the choice of extension $g'_\varepsilon$.

Now according to Lemma \ref{l:minsol}, for all
sufficiently small $\varepsilon>0$ the 
infemum $V_{\varepsilon}$
is realized by a bounded set $\Omega_{\varepsilon}\subset \mathbb{R}^2\times S^2$ containing $B_{1}(0)\times S^2$ whose boundary yields an embedded hypersurface $\Sigma_\varepsilon$. 
According to the classical regularity theory \cite{Fed} in this dimension, 
$\Sigma_{\varepsilon}$ is smooth away 
from the intersection $\Sigma_{\varepsilon}\cap\p U^{\eta_0}$. The rest of the proof is devoted to showing that the 
intersection $\Sigma_{\varepsilon}\cap\partial U^{\eta_0}$ is empty
for all sufficiently small $\varepsilon$. 

First observe that
\begin{equation}\label{e:volcomp}
\lim_{\varepsilon\to0}V_{\varepsilon}=0.
\end{equation}
Indeed, for small distances $\eta$, one may consider the boundary of the neighborhood $U^\eta$ and compute
\begin{equation}\label{e:volcompest}
\vol_{g_\varepsilon}(\partial U^\eta)\leq C_1\eta^2(1+\varepsilon\eta^{-1})^4
\end{equation}
for some constant $C_1>0$ depending only on $g$ and the background metric $h$.
In particular, $\vol_{g_\varepsilon}(\partial U^{\sqrt{\varepsilon}})\to 0$ as $\varepsilon\to0$
and (\ref{e:volcomp}) follows.

Now let $\eta_1\in(0,\eta_0)$. We claim that there exists $\varepsilon_0>0$ so that
\begin{equation}
\Sigma_{\varepsilon}\cap(U^{\eta_0}\setminus U^{\eta_1})=\emptyset
\end{equation}
for all $\varepsilon\leq\varepsilon_0$.
For sake of contradiction, assume there is a sequence $\varepsilon_j\to0$ so that 
$\Sigma_{\varepsilon_j}\cap(U^{\eta_0}\setminus U^{\eta_1})\neq\emptyset$ for all $j=1,2,\dots$.
For each $j$, choose a point
\[
x_j\in\Sigma_{\varepsilon_j}\cap(U^{(\eta_0-\eta_1)/2}\setminus U^{\eta_1}).
\] 
Let $B^h_\eta(x)$ denote the metric ball about $x$ of radius $\eta$ using the metric $h$.
Since $g_\varepsilon\to g$ smoothly on the region $U^{\eta_0}\setminus U^{\eta_1}$ where $g$ is bounded in $C^k$,
$B^h_{(\eta_0-\eta_1)/2}(x_j)\cap\Sigma_{\varepsilon}$ are smooth minimal 
hypersurfaces with respect to metrics whose coefficients and their derivatives are uniformly bounded in, say, normal coordinates. In this context, the monotonicity theorem for minimal submanifolds provides a lower bound
\begin{equation}
\vol_{g_\varepsilon}(B^h_{(\eta_0-\eta_1)/2}(x_j)\cap\Sigma_{\varepsilon})\geq V_0
\end{equation} 
where $V_0$ is a positive constant independent of sufficiently small $\varepsilon$, contradicting (\ref{e:volcomp}). This finishes the proof of Claim \ref{c:horizon}.
\end{proof}

During the remainder of the proof of Theorem A we fix some $\varepsilon$ small enough to apply Claim \ref{c:horizon}.
The parameter $\varepsilon$ will no longer play a significant role and we only mention it where necessary.
By Claim \ref{c:horizon}, we obtain a stable-minimal hypersurface $\Sigma$ which separates $M$
into two components: $A^{in}$ containing $S$ and $A^{out}$ containing $\p U^{\eta_0}$. We
mention that $\Sigma$ will generally be disconnected, having as many components as $S$. 
Let $h'=g_\varepsilon|_{\Sigma}$ denote the restriction metric.
More notation: let $M_0$ denote the compact manifold $M\setminus A^{in}$.
Abusing notation, we continue to denote the restriction of $g_\varepsilon|_{M_0}$ by $g_\varepsilon$.

Notice that $(M_0,g_\varepsilon)$ is a psc manifold whose boundary, $\Sigma$, has vanishing mean curvature.
It follows that $(M_0,g_\varepsilon)$ is Yamabe positive.
Meanwhile, since $\Sigma$ is a closed stable-minimal hypersurface with trivial normal bundle in a 
psc manifold, the classical observation of Schoen-Yau
\cite[Theorem 1]{SY79} implies that $(\Sigma,h')$ is Yamabe positive.
This allows us to find a psc metric $\tilde{h}\in[h']$.
By Theorem \ref{t:AB}, we may find 
a new psc metric $\tilde{g}_0$ on $M_0$ which has the product structure
\begin{equation}
\tilde{g}_0=\tilde{h}+dt^2
\end{equation}
near $\Sigma$. 
Now we are prepared to begin attaching psc 
null-cobordisms to the boundary of $(M_0,\tilde{g}_0)$.

Consider the closed psc manifold $(\Sigma,\tilde{h})$ and the restriction $f|_\Sigma:\Sigma\to S$
where $f$ is the retraction $f:U^{\eta_0}\to S$. By Theorem B, 
we can find a psc null-cobordism $(W,g_W)$ of $(\Sigma,\tilde{h})$ and 
an extension $\overline{f}:W\to S$. We will require $\overline{f}$ later. For now, we can form 
\[ 
\overline{M}=M_0\bigcup_{\Sigma} W,\quad\quad \bar{g}=\begin{cases}
\tilde{g}_0&\text{ on }M_0\\
g_{W}&\text{ on }W.
\end{cases}
\]
Notice that $\overline{M}$ is a smooth closed oriented manifold and $\overline{g}$ is a smooth psc metric.
\begin{claim}\label{c:en}
There is a degree-$1$ map $F:\overline{M}\to M$.
\end{claim}
\begin{proof}
Though this map is more easily understood with a visual aid, 
see Figure \ref{f:theoremA}, we describe it in some detail.
Consider the region $W'=A^{out}\cup W\subset \overline{M}.$
Also consider the annular region $U^{\eta_2}\setminus U^{\eta_0}$ surrounding $W_i'$, where $\eta_2>\eta_0$
has been chosen so that the retraction $f:U^{\eta_0}\to S$
extends to a retraction $U^{\eta_2}\to S$. It follows that we can find a map
\[
p:\overline{U^{\eta_2}\setminus U^{\eta_0}}\to \overline{U^{\eta_2}}
\]
which equals the identity on $\p U^{\eta_2}$ and equals this retraction on $\p U^{\eta_0}$.
Now consider the continuous map
\begin{equation}
F:\overline{M}\to M,\quad\quad P=\begin{cases}
\mathrm{Id}&\text{ on } \overline{M}\setminus \left(W'\cup(\overline{U^{\eta_2}\setminus U^{\eta_0}})\right)\\
p&\text{ on }\overline{U^{\eta_2}\setminus U^{\eta_0}}\\
f&\text{ on } A^{out}\\
\overline{f}&\text{ on }W.
\end{cases}
\end{equation}
Since $F$ equals the identity far away from $W$, $F$ is degree $1$.
\end{proof}
We observe that, inspecting the proof of Theorem \ref{t:AB} found in \cite{AB}, the above metric $\tilde{g}_0$
can be chosen to be conformal to $g_\varepsilon$ away from a neighborhood 
containing $\Sigma$ of any desired size. 
From this and our construction in Claim \ref{c:en}, it follows that one can arrange for $F$ to be a conformal 
diffeomorphism away from $F^{-1}(U^{\eta_2})$. This finishes the proof of Theorem A.

\section{Proof of Corollaries A and B}\label{s:cors}
We are now ready to prove Corollaries A and B.

\noindent{\bf{Proof of Corollary A}:}
Let $M$ and $S$ be as in Corollary A.
For sake of contradiction, suppose there exists $g$, an $L_E^\infty(M)\cap C^\infty(M\setminus S)$ Riemannian metric 
with $R_g\geq0$ and $Ric_g\not\equiv0$ on $M\setminus S$. 
By making use of the analysis contained in \cite{LM}, it suffices to obtain a contradiction 
under the extra assumption that $R_g>0$ on $M\setminus S$. Let us briefly explain this reduction.
Since $Ric_g\not\equiv0$, one can make a small non-conformal perturbation to obtain 
a metric $g'$ with the same regularity
as $g$ such that $R_{g'}\geq0$ and $R_{g'}(p)>0$ for some point $p\in M\setminus S$, see the proof 
of \cite[Theorem 1.7]{LM}. 
Then one can find a conformal
transformation of $g'$ to a metric $g''$ of the same regularity which has positive scalar 
curvature on $M\setminus S$, see \cite[Corollary 4.2]{LM}.
Therefore, we will assume $R_g>0$ on $M\setminus S$.

By applying Theorem A, we obtain a smooth oriented psc manifold $(\overline{M},\overline{g})$
with a degree-$1$ map to $M$. According to Proposition \ref{p:en}, $\overline{M}$ is enlargeable.
However, the existence of the psc metric $\overline{g}$ contradicts Theorem \ref{t:en}, completing
the proof of Corollary A.

\vspace{.1in}

\noindent{\bf{Proof of Corollary B}}: Let $(M,g)$ and $S$ be as in Corollary B. Notice that, as a smooth manifold,
we may regard $M$ as a closed oriented manifold $M'$ with a point removed i.e. $M\cong M'\setminus\{x_0\}$.

For sake of contradiction,
suppose that $m(M,g)<0$. We will obtain a contradiction to Theorem A by adapting
an argument first established by Lohkamp in \cite{Loh} to our low-regularity situation
using some tools from \cite{LM}. We proceed in several steps.
Our argument is very similar to the proof of \cite[Proposition 6.1]{Loh}, but we present it in order
to point out the modifications necessary in our setting.
Fix a compact set $K\subset M$ which contains the singular set $S$
and so that $M\setminus K$ is diffeomorphic to $\mathbb{R}^4\setminus \overline{B_1^4}$
with the metric decay as in Definition \ref{d:ALE}. Let $r$ denote the radial coordinate on 
$\mathbb{R}^4\setminus \overline{B_1^4}$.

The first step is to conformally change $g$ to a scalar-flat metric with negative mass. 
To begin, we may find a $W^{1,2}_{\mathrm{loc}}(M)$ solution $u>0$ satisfying
\[
\begin{cases}
\mathcal{L}_gu=0\\
\lim_{r\to\infty}u=1,
\end{cases}
\]
see the proof of \cite[Theorem 1.9]{LM}.
It follows that $0<u<1$ and $u\in C^{0,\alpha}(M)\cap C^\infty(M\setminus S)$ by the maximum principle
and standard elliptic estimates. According to \cite{Bart}, $u$ has the asymptotic expansion
\[
u(x)=1+Ar^{-2}+\mathcal{O}(r^{-3})
\]
where $A\leq0$ is a constant.
Assembling the above facts, the conformal metric $g_1=u^{2}g$ has the same regularity as $g$, is scalar-flat on $M\setminus S$,
is still asymptotically flat, and has mass $m(M,g_1)=m(M,g)+A<0$.

Next, we perturb $g_1$ to make the end conformally Euclidean while retaining the negativity
of the mass. Consider a function $\psi:M\to\mathbb{R}$
such that $\psi\equiv1$ on $K$, $\psi\equiv 0$ for $r$ sufficiently large in the asymptotic region,
and $0\leq\psi\leq1$. Set $g_\psi=\psi g_1+(1-\phi)\delta$ where $\delta$ is the flat Euclidean metric on
$M\setminus K$. By \cite[Proposition 4.1]{Schoen}, $\psi$ can be chosen so that, upon finding a solution $u'>0$ to
\[
\begin{cases}
\underline{\mathcal{L}_{g_\psi}}u'=0\\
\lim_{r\to\infty}u'=1,
\end{cases}
\]
the metric $g_2=(u')^2g_\psi$ is scalar-flat, asymptotically Euclidean, and still has negative mass. The metric
$g_2$ has the same regularity as $g$ for the same reason $g_1$ had this regularity.
Notice that $g_2$ is conformally flat on $\{r\geq R_0\}\subset (M\setminus K)$ for sufficiently large $R_0$.
It follows that $g_2=\phi^2\delta$ on $\{r\geq R_0\}$ for some function $\phi:\mathbb{R}^4\setminus B_{R_0}^4\to\mathbb{R}$.

Now we will perturb $g_2$ to make it flat at infinity while at the same time making scalar curvature positive somewhere.
Since $g_2$ is scalar-flat, $\phi$ is harmonic with respect to $\delta$ and tends to $1$ at infinity.
It follows that we may write
\[
\phi=1+\frac{m(M,g_2)}{12 r^2}+f
\]
where $f$ is a function satisfying $|f|\leq Cr^{-3}$ for some constant $C$.
By \cite[Lemma 6.2]{Loh}, in this setting one can explicitly construct a conformal factor
$\phi'>0$ so that
\[
\begin{cases}
\phi'\equiv\phi &\text{ on }\{R_1>0\}^C\\
\phi'\equiv\text{const.}&\text{ on }\{R_2>0\}\\
\Delta_\delta\phi'\leq0&\text{ on }M\\
\Delta_\delta\phi(x)<0&\text{ for some point }x\in M,
\end{cases}
\]
where $R_1<R_2$ are two radii larger than $R_0$. 

Now, the metric $g_3=(\phi')^2g_\psi$ is flat on $\{R_2>0\}$, scalar non-negative everywhere, and not scalar-flat. The metric $g_3$
has the same regularity as $g_2$ since we have not altered it within $K$.
By appropriately identifying opposing faces of the cubical complex $\p([-R_2,R_2]^{\times 4})$
in $M\setminus K$, $g_3$ descends to a metric on $T^4\# M'$. Since the map $T^4\# M'\to T^4$ 
collapsing the $M'$ factor to a point has degree $1$, Proposition \ref{p:en} implies that 
$T^4\# M'$ is enlargeable. Since $g_3$ is not scalar-flat, this
contradicts Corollary A, completing the proof of Corollary B.

\appendix
\section{ }
In this section we will provide some details on the minimization problem found in Claim \ref{c:horizon} during the proof of Theorem A. The argument we present here is only a slight modification of the proof of \cite[Lemma 6.1]{LM}.

Some notation: if $(M,g)$ is a Riemannian manifold and $k$ is some whole number, 
we write $\mathcal{H}_g^{k}$ for the $k$-dimensional Hausdorff measure associated with $g$.

\begin{lemma}\label{l:minsol}
Suppose $g'$ is a smooth metric on $\mathbb{R}^2\times S^2$ which satisfies
\begin{equation}\label{e:uniformity}
\lambda^{-1}(g_{\mathbb{H}^2}+g_{S^2})\leq g'\leq \lambda(g_{\mathbb{H}^2}+g_{S^2})
\end{equation}
for a positive constant $\lambda$. Then there exists a radius $R$ so that the minimization problem
\begin{equation}\label{e:minprob}
V=\inf\{\mathcal{H}^{3}_{g'}(\p \Omega)\colon \Omega\subset\mathbb{R}^2\times S^2\text{ open, containing }B_1(0)\times S^2\}
\end{equation}
is solved by an open set lying within $B_R(0)\times S^2\subset\mathbb{R}^2\times S^2$.
\end{lemma}

\begin{proof}
A solution to problem \eqref{e:minprob} is known as an {\emph{outer minimizing hull}} of the region $B_1(0)\times S^2$. According to a result of Fogagnolo-Mazzieri \cite[Theorem 1.1]{FM}, such an outward minimizing hull exists and is given by a bounded subset of $\mathbb{R}^2\times S^2$ so long as $(\mathbb{R}^2\times S^2,g')$ satisfies the following Euclidean isoperimetric inequality: there is a constant $C>0$ so that
\begin{equation}\label{e:Eiso}
\mathcal{H}^4_{g'}(\Omega)\leq C\mathcal{H}^{3}_{g'}(\partial\Omega)^{\frac{4}{3}}
\end{equation}
for all bounded regions $\Omega\subset \mathbb{R}^2\times S^2$ with smooth boundary. As such, we aim to establish \eqref{e:Eiso}. 

Leveraging the uniformity assumption \eqref{e:uniformity}, it suffices to establish an Euclidean isoperimentric inequality for the product manifold $M_*:=(\mathbb{R}^2\times S^2,g_{\mathbb{H}^2}+g_{S^2})$. To this end, suppose we are given an open set $\Omega\subset M_*$ with smooth boundary. When $\Omega$ has sufficiently small volume, the inequality \eqref{e:Eiso} follows from a classical argument using the fact that $M_*$ satisfies a Ricci curvature lower bound $\mathrm{Ric}^{M_*}\geq -g$ and a noncollapsing condition $\mathrm{Vol}_{M_*}(B_1(p))\geq1$. Namely, the result \cite[Lemma 3.2]{Heb} implies there is an $\eta>0$ and constant $C$ so that any $\Omega\subset M_*$ with $\mathcal{H}^4(\Omega)\leq \eta$ satisfies inequality \eqref{e:Eiso}. Now assume $\mathcal{H}^4(\Omega)\geq\eta$. The result \cite[Theorem 6.19]{Gromovbook} states that such $\Omega$ satisfy the same isoperimetric inequality as the one satisfied by the fundamental group of any manifold covered by $M_*$. Since $M_*$ covers manifolds with hyperbolic fundamental group, we conclude the existence of a $C'>0$ such that the linear inequality $\mathcal{H}^4(\Omega)\leq C'\mathcal{H}^3(\partial\Omega)$ holds, from which the Euclidean inequality \eqref{e:Eiso} quickly follows.

\end{proof}

\section*{Conflict of interest statement}
The author states that there is no conflict of interest.

\bibliography{3dPSC}

\bibliographystyle{alpha}

\end{document}